\documentclass[12pt]{amsart}

\usepackage{amssymb}

\usepackage{enumitem}

\usepackage{graphicx}

\makeatletter
\@namedef{subjclassname@2010}{%
  \textup{2010} Mathematics Subject Classification}
\makeatother

\usepackage[T1]{fontenc}

\newtheorem{theorem}{Theorem}[section]

\newtheorem{lemma}[theorem]{Lemma}

\theoremstyle{definition}
\newtheorem{definition}[theorem]{Definition}
\newtheorem{remark}[theorem]{Remark}

\numberwithin{equation}{section}

\DeclareMathOperator{\Gal}{Gal}
\DeclareMathOperator{\red}{red}
\DeclareMathOperator{\Frob}{Frob}
\DeclareMathOperator{\car}{char}
\DeclareMathOperator{\lift}{lift}
\DeclareMathOperator{\tr}{tr}
\DeclareMathOperator{\Aut}{Aut}
\DeclareMathOperator{\GL}{GL}

\newcommand{\m}{\mathfrak{m}}
\newcommand{\Q}{\mathbb{Q}}
\newcommand{\C}{\mathbb{C}}
\newcommand{\F}{\mathbb{F}}
\newcommand{\Z}{\mathbb{Z}}
\newcommand{\mat}[4]{\left(\begin{array}{cc}
#1 &#2\\
#3 &#4
\end{array}\right)}
\newcommand{\GalL}{\Phi}

\newtheorem{theorem*}[]{Theorem}
\theoremstyle{definition}
\newtheorem{eg}[theorem]{Example}

\usepackage{tikz}
\usetikzlibrary{cd}

\begin{document}


\baselineskip=17pt


\title[Wild Galois representations]{Wild Galois representations: Elliptic curves over a $3$-adic field}

\author[N. Coppola]{Nirvana Coppola}
\address{University of Bristol\\
School of Mathematics,\
Fry Building\\
Woodland Road\\
BS8 1UG Bristol, UK}
\email{nc17051@bristol.ac.uk}

\date{}



\begin{abstract}
Given an elliptic curve $E$ over a local field $K$ with residue characteristic $3$, we investigate the action of the absolute Galois group of $K$ in the case of potentially good reduction. The hardest case is when the $\ell$-adic Galois representation attached to $E$ has non-cyclic inertia image, isomorphic to $C_3 \rtimes C_4$. In this work we describe such a representation explicitly.
\end{abstract}

\maketitle

\section{Introduction}

Let $E$ be an elliptic curve defined over a field $K$, i.e. a projective smooth curve of genus $1$, with a fixed point $O$. It is given by a non-singular Weierstrass equation 
\begin{align*}
y^2 + a_1 xy + a_3 y = x^3 + a_2 x^2 + a_4 x + a_6,
\end{align*}
with coefficients $a_i \in K$. We refer to \cite[III \S 1]{Silv1} for the definition of the invariants associated to this curve, in particular we denote by $\Delta$ its discriminant. We denote by $\overline K$ a fixed algebraic closure of $K$, and let $G_K=\Gal(\overline K/K)$ be the absolute Galois group of $K$; then any $\sigma \in G_K$ acts on the points of the curve by sending the point $P$ of coordinates $(x,y)$ to $\sigma(P)=(\sigma(x),\sigma(y))$. In particular, for any integer $m$, $G_K$ acts on the group of $m$-torsion points $E[m]$ (i.e. the points of order dividing $m$). This induces, for any prime $\ell$, an action on the $\ell$-adic Tate modules, which are defined as
\begin{align*}
    T_\ell (E)= \varprojlim_n E[\ell^n].
\end{align*}

This action is called $\ell$-adic Galois representation attached to $E$, and denoted $\rho_{E,\ell}$. For $\ell$ different from the characteristic of $p$, we have that $T_\ell (E)$ is a free $\Z_\ell$-module of rank $2$; so after considering the tensor product with an algebraic closure $\overline \Q_\ell$ of $\Q_\ell$, we view $\Aut(T_\ell(E))$ as a subgroup of $\GL_2 (\overline \Q_\ell)$, and we also denote by $\rho_{E,\ell}$ the representation we obtain:
\begin{align*}
    \rho_{E,\ell} : G_K \rightarrow \GL_2(\overline \Q_\ell).
\end{align*}
Furthermore, we fix an embedding $\overline \Q_\ell \rightarrow \C$, and in what follows we identify the eigenvalues and traces of the elements of $G_K$ with the corresponding complex numbers.

Let $K$ be a local field of characteristic $0$, complete with respect to a discrete valuation $v$, with valuation ring $O_K$, maximal ideal $\m$ and perfect residue field $k$. The residue characteristic of $K$ is the characteristic of $k$, and we will always denote it by $p$. A typical example of this is $K=\Q_p$, with the $p$-adic valuation, or any algebraic extension of it. We denote by $K^{nr}$ the maximal unramified extension of $K$ contained in $\overline K$.

Recall that the inertia subgroup $I_K$ of $G_K$ is isomorphic to $ \Gal(\overline K / K^{nr})$ and, if we denote by $\overline k$ a fixed algebraic closure of $k$, we have
\begin{align*}
G_K/I_K \cong \Gal(\overline k/k).
\end{align*}

This group is procyclic, generated by the Frobenius element, i.e. the automorphism: $x \mapsto x^{|k|}$ for all $x \in \overline k$. Any preimage $\Frob$ of the Frobenius in $G_K$ under the projection to $\Gal(\overline k/k)$ is called a Frobenius element of $G_K$.

To study the $\ell$-adic Galois representation attached to an elliptic curve one has to study the Frobenius and the inertia actions and see how they are related. This representation changes substantially according to the reduction type of the curve, which can be good, multiplicative, or additive (for the definition see \cite[VII \S 5]{Silv1}); moreover in the last case, there is a finite extension of $K$ where the curve acquires either good or multiplicative reduction, so we say $E/K$ has potentially good (resp. multiplicative) reduction, as explained in Theorem \ref{silvermanVII5}. Let $p$ and $\ell$ be two distinct primes, and suppose that the field $K$ over which $E$ is defined is a $p$-adic field, i.e. isomorphic to a finite extension of $\Q_p$. Then the $\ell$-adic Galois representation attached to $E$ can easily be studied in the case of good and potentially multiplicative reduction, respectively via the Criterion of N\'eron-Ogg-Shafarevich (see \cite[VII \S 7]{Silv1} and Theorem \ref{nos}) and the theory of the Tate curve (see \cite[V \S 3]{Silv2}); the last remaining case is that of bad additive reduction with potentially good reduction. By the work of Kraus (\cite{kraus}) we have that the image of inertia is isomorphic to one of the following:
\begin{align*}
&C_2,C_3,C_4,C_6,\\
C_3 \rtimes &C_4 \text{ only when }p=3,\\
Q_8, SL_2&(\F_3)\text{ only when }p=2,
\end{align*}
where $C_n$ denotes the cyclic group of order $n$ and $Q_8$ is the quaternion group.

In this work we present a result to determine the $\ell$-adic Galois representation $\rho_{E,\ell}$ attached to an elliptic curve over a field of residue characteristic $3$, where $\ell$ is a prime different from $3$, such that the image of inertia is non-cyclic, hence isomorphic to $C_3 \rtimes C_4$. We will use the notation of \cite{Tim} for group names and presentations and for their character tables. In particular we will prove the following result.

\begin{theorem}\label{mainthm}
Let $E$ be an elliptic curve with potentially good reduction over a $3$-adic field $K$, with Weierstrass equation of the form $y^2=f(x)$ and discriminant $\Delta$. Fix a fourth root $\Delta^{1/4}$ of $\Delta$ and define $F$ to be the compositum of the splitting field of $f$ over $K$ and $K(\Delta^{1/4})$; let $F'$ be the Galois closure of $F/K$. 

Let $\chi$ be the unramified character of $G_K$ (i.e. trivial on inertia) such that
\begin{align*}
\chi(\Frob)= i^n \sqrt{3^n},
\end{align*}
and let $\psi$ be as follows. Let $n=[k:\F_3]$. If $n$ is even, let $\psi$ be the representation of $\Gal(F'/K)$, which is isomorphic to $C_3 \rtimes C_4$, with character:
\begin{align*}
\begin{array}{c|rrrrrr}
  \rm class&\rm1&\rm2&\rm3&\rm4A&\rm4B&\rm6\cr
  \rm size&1&1&2&3&3&2\cr
  \hline
  \psi&2&-2&-1&0&0&1\cr
\end{array}
\end{align*}
while if $n$ is odd then let $\psi$ be the representation of $\Gal(F'/K)$, which is isomorphic to $C_3 \rtimes D_4$, with character:
\begin{align*}
\begin{array}{c|rrrrrrrrr}
  \rm class&\rm1&\rm2A&\rm2B&\rm2C&\rm3&\rm4&\rm6A&\rm6B&\rm6C\cr
  \rm size&1&1&2&6&2&6&2&2&2\cr
\hline
  \psi&2&-2&0&0&-1&0&-i\sqrt{3}&i\sqrt{3}&1\cr
\end{array}
\end{align*}
where the presentation of $\Gal(F'/K)$ and its conjugacy classes are as in Section \ref{presentations}.

Let $\ell$ be a prime different from $3$, and let $\rho_{E,\ell}$ be the $\ell$-adic Galois representation attached to $E$. Suppose that the image of the inertia subgroup of $G_K$ is isomorphic to $C_3 \rtimes C_4$. Then $\rho=\rho_{E,\ell}$ factors as
\begin{align*}
    \rho = \chi \otimes \psi.
\end{align*}

\end{theorem}

Part of this result follows immediately from \cite[Theorem 1]{kraus}, \cite[Theorem 2]{ST} and the classification of the representations of the groups $C_3 \rtimes C_4$ and $C_3 \rtimes D_4$ in \cite{Tim}. The result for the case of odd $n$ is more subtle, as there are two possibilities for $\psi$, and we will prove that only one of these occurs, via explicit computation.

In Section \ref{preliminaries} we present the general setting for this problem. In particular we recall some results about the Galois representations attached to elliptic curves with good reduction and what is already known about the potentially good reduction case. In Section \ref{3adic} we tackle the case over $3$-adic fields when the action of inertia is non-cyclic. In Section \ref{presentations}, we fix the a presentation for the group $\Gal(F'/K)$. The proof is divided into two parts: in Section \ref{iink} we give the proof for the case of even $n$, and in Section \ref{inotink} we assume that $n$ is odd.

\section{Preliminaries}\label{preliminaries}

Let $E$ be an elliptic curve over a $p$-adic field $K$. In this section we give some general results about the study of the $\ell$-adic Galois representation attached to $E$, for $\ell \neq p$. As mentioned in the introduction, this representation varies according to the different reduction types that $E$ can have.  There are several results about reduction types in \cite[VII \S 5]{Silv1}, which can be summarised as follows.

\begin{theorem}\label{silvermanVII5}
\begin{itemize}
\item Let $K'/K$ be an unramified extension. Then the reduction types of $E/K$ and $E/K'$ are the same.
\item Let $K'/K$ be a finite extension. If $E/K$ has good or multiplicative reduction over $K$, then $E/K'$ has the same reduction type.
\item There exists a finite extension $K'/K$ over which $E$ has either good or multiplicative reduction.
\item The curve $E$ has potentially good reduction if and only if its $j$-invariant is integral.
\end{itemize}
\end{theorem}

So if $E/K$ has additive reduction, then there is a finite extension over which it acquires either good or multiplicative reduction; once it does, its reduction type does not change after a finite extension. This suggests the terminology \emph{potentially good (resp. multiplicative) reduction}. In case of additive reduction, using the last statement, we can check easily whether the reduction is potentially good or not: we only need to compute the valuation of the $j$-invariant.

\begin{definition}
A Galois representation is \emph{unramified} if it is trivial when restricted to the inertia subgroup.
\end{definition}

We recall here the main content of the Criterion of N\'eron - Ogg - Shafarevich. For the full statement and the proof see \cite[VII \S 7 Theorem 7.1]{Silv1}.

\begin{theorem}\label{nos}
Let $E/K$ be an elliptic curve. The following are equivalent:
\begin{itemize}
\item $E$ has good reduction over $K$;
\item $\rho_{E,\ell}$ is unramified for some (all) primes $\ell \neq p$.
\end{itemize}
\end{theorem}

So, if $E/K$ has good reduction, the Galois representation factors through the quotient of $G_K$ by $I_K$, which is generated by Frobenius.

\begin{center}
    \begin{tikzcd}
    G_K \arrow{rr}{\rho_{E,\ell}} \arrow{dr}&   & \Aut(T_\ell(E))\\
        &G_K/I_K \cong \langle \Frob \rangle \arrow{ur}&
    \end{tikzcd}
\end{center}

In this case the problem reduces to the study of the image of Frobenius. As it is shown in \cite[IV \S 2.3]{Serre}, it is always diagonalisable (at least in an algebraic closure of $\Q_\ell$), so it is enough to compute its eigenvalues to uniquely determine its action.

\begin{lemma}\label{charpoly}
The characteristic polynomial of $\rho_{E,\ell}(\Frob)$ is independent of $\ell$ (as explained in \cite[Theorem 2]{ST}) and given by
\begin{align*}
    F(T)=T^2-aT+q,
\end{align*}
where $q=|k|$ and $a=q+1-|\tilde{E}(k)|$, where by $\tilde{E}$ we denote the reduction of the curve $E$ to $k$.
\end{lemma}

\begin{proof}
See \cite[IV \S 1.3]{Serre}.
\end{proof}

\begin{eg}\label{keyexample}
Let $K$ be a local field with residue characteristic $3$ and let $k$ be its residue field. Let $E/K$ be an elliptic curve such that the reduction over $k$ is: $y^2=x^3-x$. In particular, $E$ has good reduction over $K$.

We want to compute the eigenvalues of $\rho_{E,\ell}(\Frob)$, for a prime $\ell \neq 3$, as elements of $\C$.

First we assume that $k = \F_3$. Then the reduced curve $\tilde E$ has the following $4$ points over $\F_3$: $\{ O, (0,0), (1,0), (2,0) \}$, therefore in the notation above $a=0$, $q=3$ and the roots of $F(T)$, hence the eigenvalues of $\rho_{E,\ell}(\Frob)$, are $\pm i \sqrt{3}$. 

If $[k : \F_3] =n \geq 1$, then $\rho_{E,\ell}(\Frob)$ acts as the $n$-th power of the linear operator described above, so its eigenvalues are $(\pm i \sqrt{3})^n$. In particular we have:
\begin{center}
\begin{tabular}{c|cc}
$n \equiv 0 \pmod 4$ & $+ 3^{n/2}$ & $+ 3^{n/2}$ \\ 
$n \equiv 2 \pmod 4$ & $- 3^{n/2}$ & $- 3^{n/2}$ \\ 
$n$ odd & $i^n \sqrt{3^n}$ & $-i^n \sqrt{3^n}$ \\ 
\end{tabular} 
\end{center}

\end{eg}

Now let us consider elliptic curves with additive, potentially good reduction. Our first step is to determine the image of the inertia subgroup. As a consequence of the Criterion of N\'eron-Ogg-Shafarevich, if the curve $E/K$ has potentially good reduction then the inertia subgroup $I_K$ acts through a finite quotient on $T_\ell (E)$ for some (all) $\ell \neq p$ prime. (See \cite[VII \S 7 Corollary 7.3]{Silv1}.)

Since we are only interested in the action of $I_K=\Gal(\overline K/K^{nr})$ and since the reduction types of $E$ over $K$ and $K^{nr}$ are the same, we may work simply on $K^{nr}$. Suppose $E$ has potentially good reduction (recall this is equivalent to the $j$-invariant of $E$ being integral) and let $L$ be the minimal finite extension of $K^{nr}$ over which $E$ acquires good reduction. Then it follows from Theorem \ref{nos} that for $\ell \neq p$ the Galois representation $\rho_{E,\ell}$ factors through the quotient $I_K/I_L$, which is isomorphic to $\Gal(L/K^{nr})$. Indeed, the subgroup $I_L$ is normal in $I_K$, so $L/K^{nr}$ is Galois. This group is also finite, and since we chose $L$ to be minimal, it injects into $\Aut(T_{\ell}(E))$.

\begin{center}
\begin{tikzcd}
\Gal(\overline K/K^{nr}) \arrow{rr}{\rho_{E,\ell}} \arrow{dr} && \Aut(T_\ell(E))\\
	&\Gal(L/K^{nr}) \arrow[hook]{ur}
\end{tikzcd}
\end{center}

Furthermore, there is an injection of $\Gal(L/K^{nr})$ into $\Aut(\tilde{E_L})$, where $\tilde{E_L}$ is the reduction of a minimal equation for $E$ over $L$. For more details see \cite[proof of Theorem 2]{ST}. In particular, the image of inertia does not depend on $\ell$, and we can restrict the set of possible inertia groups to
\begin{align*}
&C_2,C_3,C_4,C_6,\\
C_3 \rtimes &C_4 \text{ only when }p=3,\\
Q_8, SL_2&(\F_3)\text{ only when }p=2.
\end{align*}

This list comes from the following classification of the automorphisms of an elliptic curve defined over a field of characteristic $p$ (see \cite[III \S 10 proof of Theorem 10.1 and Appendix A Proposition 1.2(c), Exercise A.1]{Silv1}).

\begin{center}
\begin{tabular}{c|c|c|c}
	& $j \neq 0,1728$ & $j=1728$ & $j=0$\\
\hline
$p \neq 2,3$ & $C_2$ & $C_4$ & $C_6$\\
\hline
$p=3$ & $C_2$ & $C_3 \rtimes C_4$ & $C_3 \rtimes C_4$\\
\hline
$p=2$ & $C_2$ & $SL_2(\F_3)$& $SL_2(\F_3)$
\end{tabular}
\end{center}

Since these groups have all different orders, if we know the degree of the extension $L/K^{nr}$, we can uniquely determine the Galois group of this extension, hence the structure group of the image of inertia. From this moment on, we denote by $\GalL$ the group $\Gal(L/K^{nr})$. In \cite{kraus}, there are complete classification theorems that depend on the residue characteristic being $2$, $3$ or higher. In this work we focus only on the case $p=3$. The main lemma to compute the extension $L$ is the following.

\begin{lemma}\label{corp3}
With the notations as above, we have
\begin{align*}
L=K^{nr}(E[2],\Delta^{1/4}),
\end{align*}
where $\Delta^{1/4}$ is any fourth root of $\Delta$.
\end{lemma}

\begin{proof}
See the Corollary to \cite[Lemma 3]{kraus}.
\end{proof}

Using this result and Tate's Algorithm (see \cite[IV \S 9]{Silv2}), Kraus proves the following classification theorem.

\begin{theorem}\label{p3}
For $p=3$ we have the following classification of the possible images of inertia:
\begin{itemize}
\item if $E$ has type $I^*_0$, then $v(\Delta)=6$ and $\GalL \cong C_2$;
\item if $E$ has type $III$, then $v(\Delta)=3$ and $\GalL \cong C_4$;
\item if $E$ has type $III^*$, then $v(\Delta)=9$ and $\GalL \cong C_4$;
\item if $v(\Delta) \equiv 0 \pmod 4$, then $\GalL \cong C_3$;
\item if $v(\Delta) \equiv 2 \pmod 4$ and $E$ has type different from $I^*_0$, then $\GalL \cong C_6$;
\item if $v(\Delta)$ is odd and $E$ has type different from $III$ and $III^*$, then $\GalL \cong C_3 \rtimes C_4$.
\end{itemize}
\end{theorem}

We refer to \cite[Theorems 2 and 3]{kraus} for the classification theorem in the case $p=2$. It is worth mentioning a more general result, which can be found in \cite[\S 2 Corollary 3]{ST}.

\begin{theorem}\label{serretate}
We have $L=K^{nr}(E[m])$, where $m$ is any integer with $m \geq 3$ and $(p,m)=1$. The Galois group $\Gal(\overline K/L)$ is equal to $\ker (\rho_{E,\ell})$ for any $\ell \neq p$.
\end{theorem}

\section{Curves over a $3$-adic field with non-cyclic inertia action}\label{3adic}

In the following, we keep the notation and assumptions from the previous sections. When $p \geq 5$, the extension $L/K^{nr}$ is tamely ramified. Hence, for all $\ell \neq p$ prime, the Galois representation $\rho_{E,\ell}$ can be fully described, with the approach in \cite{weilrep}. A similar approach also works when $p=3$ but the image of inertia is cyclic and tame. Here we consider the case when $p=3$ and $E$ has image of inertia isomorphic to $C_3 \rtimes C_4$, i.e. the last case in Theorem \ref{p3}, and we present an algorithm to describe completely the Galois action.

\subsection{Setup}\label{presentations}

Let $k$ be the residue field of $K$, and let $n=[k:\F_3]$. Then $n$ is even if and only if $K$ contains a primitive fourth root of unity, which we denote $\zeta_4$ (see \cite[XIV \S 3 Lemma 1]{SerreLocFields}). The Galois representation changes substantially in these two cases. Let us fix some notation. By Lemma \ref{corp3}, we have that the minimal field extension $L/K^{nr}$ where $E$ acquires good reduction is generated by $2$-torsion and a $4$-th root of the discriminant. Let $F$ be the field extension of $K$ generated by these elements. Note that since $\car (K) =0$, the curve $E$ can be written with a Weierstrass equation of the form $y^2=f(x)$, where $f$ is a monic polynomial of degree $3$. If $\alpha,\beta,\gamma$ are the roots of $f$ in $\overline K$, $\Delta$ the discriminant of $E$, then $F=K(\alpha,\beta,\gamma,\Delta^{1/4})$, for some choice of $\Delta^{1/4}$. Then the discriminant $\Delta_f$ of the defining polynomial $f$, which is given by
\begin{align*}
\Delta_f = (\beta-\alpha)^2 (\gamma - \beta)^2 (\alpha- \gamma)^2,
\end{align*}
differs from $\Delta$ only by a factor $16$; in particular we have that $\sqrt{\Delta_f}$ and therefore $\sqrt{\Delta}$ are in the field generated by $2$-torsion.
\begin{remark}
The extension $F/K$ is totally ramified of degree $12$. Indeed, since $L=FK^{nr}$, $L/F$ is unramified and $L/K^{nr}$ is totally ramified of degree $12$, we have that $F/K$ has a subextension which is totally ramified of degree $12$. On the other hand, we have $[F:K] \mid 12$, since the extension of $K$ generated by $2$-torsion has degree dividing $6$ and it contains a square root of the discriminant, so $F$ is at most a quadratic extension of it. Therefore the degree is equal to $12$ and the whole extension $F/K$ is totally ramified. Moreover, as it has good reduction over $L=FK^{nr}$, $E$ acquires good reduction over $F$.
\end{remark}
However $F/K$ is not necessarily Galois. In fact it is Galois, with Galois group isomorphic to $\GalL \cong C_3 \rtimes C_4$, if and only if $\zeta_4 \in K$, i.e. if $n$ is even. Otherwise, its Galois closure is $F(\zeta_4)$ and $\Gal(F(\zeta_4)/K)$ is isomorphic to the semidirect product $(C_3 \rtimes C_4) \rtimes C_2$. As follows from the classification in \cite{Tim}, this group is $ C_3 \rtimes D_4$. We will now fix a presentation for the group $\Gal(F(\zeta_4)/K)$ in the two cases and show that this group is isomorphic to $C_3 \rtimes C_4$ for $n$ even, $C_3 \rtimes D_4$ for $n$ odd. 

Suppose first that $n$ is even. With the notation used above, we define $\sigma$ and $\tau$ to be the generators of $\Gal(F/K)$ that act on $\alpha,\beta,\gamma$ and $\Delta^{1/4}$ as follows:
\begin{align*}
    \sigma:& \alpha \mapsto \beta, \quad \beta \mapsto \gamma, \quad \gamma \mapsto \alpha, \qquad \Delta^{1/4} \mapsto \Delta^{1/4};\\
    \tau:& \alpha \mapsto \alpha, \quad \beta \mapsto \gamma, \quad \gamma \mapsto \beta, \qquad \Delta^{1/4} \mapsto \zeta_4 \Delta^{1/4}.
\end{align*}

Then $\Gal(F/K)$ has the following presentation
\begin{align*}
\langle \sigma,\tau; \sigma^3 = \tau^4 = 1, \tau\sigma\tau^{-1} = \sigma^{-1}\rangle. 
\end{align*}

This group is isomorphic to $C_3 \rtimes C_4$, which is given by the presentation in \cite{Tim}, via the isomorphism from $C_3 \rtimes C_4$ to $\Gal(F/K)$ defined by $a \mapsto \sigma \tau^2$ and $b \mapsto \tau$; in fact it is easy to check that these elements satisfy $a^6=1, b^2=a^3, bab^{-1}=a^{-1} $.

Suppose now that $n$ is odd, i.e. $\zeta_4 \notin K$. Then the Galois closure of $F/K$ is given by $F(\zeta_4)$ and the subgroup generated by the elements $\sigma,\tau$ is the inertia subgroup of $\Gal(F(\zeta_4)/K)$. Furthermore, $\Gal(F(\zeta_4)/K)$ contains an extra unramified automorphism corresponding to the map $\zeta_4 \mapsto -\zeta_4$, which we call $\phi$. Therefore it is presented by the following relations:
\begin{align*}
\sigma^3=1;\;
\tau^4=1;\;
\phi^2=1;\;
    \sigma \phi= \phi \sigma;\;
    \phi \tau \phi = \tau^{-1};\;
    \tau \sigma \tau^{-1}= \sigma^2.
\end{align*}

Now $C_3 \rtimes D_4 = \langle a,b,c | a^3=b^4=c^2=1, bab^{-1} = cac =a^{-1}, cbc=b^{-1} \rangle$ (see \cite{Tim}). The map from $C_3 \rtimes D_4$ to $\Gal(F(\zeta_4)/K)$ given by $a \mapsto \sigma$, $b \mapsto \tau$ and $c \mapsto \tau \phi$ is an isomorphism.

We denote the conjugacy classes of the group obtained in the two cases as follows (for the sake of completeness, we will use both the notation of \cite{Tim} and the one introduced in this paper):
\begin{itemize}
\item if $n$ is even, the conjugacy classes of $\Gal(F/K) \cong C_3 \rtimes C_4$ are $1=[ e ], \ 2 = [\tau^2=b^2],\ 3=[\sigma=ab^2], \ 4A=[\tau=b], \ 4B = [\sigma \tau = ab^{-1}],\ 6=[\sigma \tau^2=a]$;
\item if $n$ is odd, the conjugacy classes of $\Gal(F(\zeta_4)/K) \cong C_3 \rtimes D_4$ are $1=[ e ], \ 2A = [\tau^2=b^2],\ 2B =[\phi=b^{-1}c], \ 2C =[\tau \phi=c],\ 3=[\sigma=a], \ 4=[\tau=b], \ 6A=[\sigma \phi=ab^{-1}c], \ 6B =[\sigma^2 \phi=a^2b^{-1}c],\ 6C=[\sigma \tau^2=ab^2]$.
\end{itemize}

We can now prove Theorem \ref{mainthm}.

\subsection{The case $n$ is even}\label{iink}
\begin{proof}[Proof of Theorem \ref{mainthm}]
Let us first consider the case $n$ even, i.e. $\zeta_4 \in K$. Then, if $F$ is as at the beginning of Section \ref{3adic}, we showed that $F/K$ is Galois with Galois group $C_3 \rtimes C_4$. But then $F^{nr}/K$ is the compositum of the Galois extensions $F/K$ and $K^{nr}/K$, which intersect in $K$ since $F/K$ is totally ramified and $K^{nr}/K$ is unramified. So $\Gal(F^{nr}/K) \cong \Gal(F/K) \times \Gal(K^{nr}/K)$. In particular the Frobenius element, that generates $\Gal(K^{nr}/K)$, commutes with every element of this group, therefore its image under $\rho$ can be represented as a scalar matrix. We computed in Example \ref{keyexample} the eigenvalues of the Frobenius element of $F$, which coincide and are equal to $(-3)^{n/2}$ for every even $n$. As $F/K$ is totally ramified, the image of the Frobenius element of $K$ under $\rho$ is conjugate, so it has the same eigenvalues.

Now define the following unramified character:
\begin{align*}
    \chi(\Frob)&=(-3)^{n/2};\\
    \chi\big|_{I_K}&=1.
\end{align*}

Then $\rho(\Frob)=\chi(\Frob) Id$ and there exists a representation $\psi$ such that $\rho=\chi \otimes \psi$. The representation $\psi$ is irreducible of dimension $2$, since $\rho$ is, it is trivial on Frobenius and coincides with $\rho$ on inertia; therefore it factors through $C_3 \rtimes C_4$ and, as a representation of this finite group, it is faithful. The group $C_3 \rtimes C_4$ has only one irreducible faithful $2$-dimensional representation (see \cite{Tim}), so the Galois representation is completely described by it. Namely the character of $\psi$ in this case is:

$$
\begin{array}{c|cccccc}
  \rm class&\rm1&\rm2&\rm3&\rm4A&\rm4B&\rm6\cr
  \rm size&1&1&2&3&3&2\cr
\hline
  \psi&2&-2&-1&0&0&1\cr
\end{array}
$$
as claimed. \end{proof}

\subsection{The case $n$ is odd}\label{inotink}

\begin{lemma}\label{twofields}
Let $F(\zeta_4)$ be the Galois closure of $F$. Then:
\begin{align*}
    F(\zeta_4)=K(\sqrt{\beta-\alpha}, \sqrt{\gamma-\beta}, \sqrt{\alpha-\gamma}, \sqrt{\alpha-\beta}, \sqrt{\beta-\gamma}, \sqrt{\gamma-\alpha}).
\end{align*}
\end{lemma}

\begin{proof}
Let
\begin{align*}
    F'=K(\sqrt{\beta-\alpha}, \sqrt{\gamma-\beta}, \sqrt{\alpha-\gamma}, \sqrt{\alpha-\beta}, \sqrt{\beta-\gamma}, \sqrt{\gamma-\alpha}).
\end{align*}
First of all, we prove that $F(\zeta_4) \subseteq F'$. Indeed $\alpha, \beta, \gamma$ are clearly in $F'$; $\dfrac{\sqrt{\beta-\alpha}}{\sqrt{\alpha-\beta}}$ is a primitive fourth root of unity contained in $F'$; finally one possible choice for $\Delta^{1/4}$ is given by the product $2\sqrt{(\beta-\alpha)(\gamma-\beta)(\alpha-\gamma)}$, which is in $F'$. Since $F'$ contains this element and a primitive fourth root of unity, then also all the other fourth roots of $\Delta$ (which are in $F(\zeta_4)$) are in $F'$. To prove that $F' \subseteq F(\zeta_4)$, let $B=K(\alpha,\beta,\gamma)$; we show that $[F':B] \mid [F(\zeta_4):B]$. The extension $F(\zeta_4)/B$ is of degree $4$, with an unramified subextension of degree $2$ and a totally tamely ramified subextension of degree $2$. Therefore $\Gal(F(\zeta_4)/B) \cong C_2 \times C_2$. The extension $F'/B$ is the compositum of some quadratic extensions, so it is abelian of exponent $2$. By \cite[XIV \S 4 Exercise 3]{SerreLocFields}, we have $|B^\times/(B^\times)^2| = 4$, hence $B^\times/(B^\times)^2 \cong C_2 \times C_2$, and by Kummer theory, the abelian extensions of $B$ of exponent $2$ are in bijection with the subgroups of $B^\times/(B^\times)^2$, which are five, namely $B$, three quadratic extensions and the biquadratic; therefore $[F':B] \mid 4$.
\end{proof}

We now want to compute the action of $\phi$ on the generators of $F'$; for any choices of the square roots, we know that $\phi(\sqrt{\beta-\alpha})=\pm \sqrt{\beta-\alpha}$ and similarly
$\phi(\sqrt{\alpha-\beta})= \pm \sqrt{\alpha-\beta}$. On the other hand, $\phi$ changes the sign of $\dfrac{\sqrt{\alpha-\beta}}{\sqrt{\beta-\alpha}}$ for it is a primitive 
fourth root of unity. 
Therefore, we have either:
\begin{itemize}
\item $\phi(\sqrt{\beta-\alpha})=\sqrt{\beta-\alpha}$ and $\phi(\sqrt{\alpha-\beta})=-\sqrt{\alpha-\beta}$, or
\item $\phi(\sqrt{\beta-\alpha})=-\sqrt{\beta-\alpha}$ and $\phi(\sqrt{\alpha-\beta})=\sqrt{\alpha-\beta}$.
\end{itemize}
Without loss of generality the first condition holds, so $\sqrt{\beta-\alpha} \in F$. Similarly, using the relations between the generators of $\Gal(F'/K)$, we have that $\phi$ fixes $\sqrt{\gamma-\beta},\sqrt{\alpha-\gamma}$ and changes sign to the other generators of $F'$; therefore $F$, which is the subfield of $F'$ fixed by $\phi$, satisfies $$F=K(\sqrt{\beta-\alpha},\sqrt{\gamma-\beta},\sqrt{\alpha-\gamma}).$$

\begin{lemma}\label{reduction_wild}
Let $O_F$ be the ring of integers of $F$, with maximal ideal $\m_F$. Then with the same notation as above, we have
\begin{align*}
    \dfrac{\sigma(x)}{x} \equiv 1 \pmod {\m_F},
\end{align*}
for all $x \in O_F \setminus \{0\}$.
\end{lemma}

\begin{proof}
As $\sigma$ is in the wild inertia subgroup of $\Gal(F'/K)$, which is equal to the first ramification group by \cite[IV \S 2 Corollary 1 to Proposition 7]{SerreLocFields}, we have $\sigma(x) \equiv x \pmod {\m_F^2}$. If $x \in O_F^\times$ (i.e. if $x$ is a unit), then $\sigma(x)/x \equiv 1 \pmod {\m_F^2}$, hence modulo ${\m_F}$; if $x=\pi$ is a uniformiser of $O_F$ of $F$ then by \cite[IV \S 2 Proposition 5]{SerreLocFields} we have $\sigma(x)/x \equiv 1 \pmod{\m_F}$. In general $x=\pi^a u$ where $a$ is a non-negative integer and $u \in O_F^\times$, so $\sigma(x)/x = (\sigma(\pi)/\pi)^a  \sigma(u)/u \equiv 1 \pmod{\m_F}$. 
\end{proof}

\begin{lemma}\label{minmod}
Let $E$ be as before. Then the reduction of some minimal model for $E/F$ on the residue field is
\begin{align*}
    \tilde{E}/k : y^2=x^3-x.
\end{align*}
\end{lemma}

\begin{proof}
First note that we can write, over $F$, the equation for $E$ as follows:
\begin{align*}
    y^2 = (x-\alpha)(x-\beta)(x-\gamma).
\end{align*}

Now, operating the change of variables (well-defined over $F$)
\begin{align*}
    \left\lbrace \begin{array}{cl}
        x =&x'(\beta-\alpha)+\alpha  \\
        y =&y'(\sqrt{\beta-\alpha})^3
    \end{array}\right.
\end{align*}
we obtain the new equation
\begin{align*}
(y')^2 = x'(x'-1)(x'-\lambda),    
\end{align*}
where $\lambda = \dfrac{\gamma - \alpha}{\beta - \alpha}$. Finally, note that $\alpha - \gamma= \sigma^{2} (\beta-\alpha)$, and since $\sigma$ is an element of the wild inertia subgroup of $\Gal(F'/K)$ then by Lemma \ref{reduction_wild}, $\dfrac{\sigma^{2} (\beta-\alpha) }{ \beta - \alpha} \equiv 1 $ in the residue field. So the reduction in $k$ of $\lambda$ is the same as the reduction of $-\dfrac{\sigma^2(\beta-\alpha)}{\beta-\alpha}$, i.e. $-1$. With simplified notation, the reduction of $E$ in $k$ is therefore $y^2=x^3-x$.
\end{proof}

\begin{proof}[Proof of Theorem \ref{mainthm}]
In Example \ref{keyexample} we computed the eigenvalues of the Frobenius element of $F$, and hence of $K$ (as $F/K$ is totally ramified), acting on $E$, which are $\pm i^n \sqrt{3^n}$. In particular, since $n$ is odd, they are complex conjugate and the trace of Frobenius is $0$. 

Let $\chi$ be the following unramified character of $G_K$:
\begin{align*}
    \chi(\Frob)& =i^n \sqrt{3^n};\\
    \chi_{\big|{I_K}}& = 1.
\end{align*}

Then $\rho(\Frob)=\chi(\Frob)\mat{1}{0}{0}{-1}$ and $\rho(\Frob^2)=\chi(\Frob)^2 Id$. Let $F_2$ be the field extension of $K$ fixed by $\Frob^2$: then it is an unramified extension of $K$ of degree $2$, i.e. $F_2=K(\zeta_4)$. Also, in the notation used above, $F'=F(\zeta_4)= F F_2$. So the Galois group described before $\Gal(F(\zeta_4)/K)$ is generated by $\sigma,\tau$ and the class of $\Frob$ modulo $\Frob^2$, which we can identify with $\phi$. Moreover, there exists an irreducible representation $\psi$ of $G_K$ such that
\begin{align*}
    \rho = \chi \otimes \psi;
\end{align*}
to find it, since $\chi$ is a character, it is sufficient to consider $\psi(g)=\dfrac{1}{\chi(g)} \rho(g)$. The kernel of this representation $\psi$ is precisely $\Gal(\overline K /F(\zeta_4))$, so $\psi$ factors through the finite group $\Gal(F(\zeta_4)/K)$ and it is indeed an irreducible faithful representation of the finite group $C_3 \rtimes D_4$.

By looking at the character table of $C_3 \rtimes D_4$ (again, see \cite{Tim}) it follows that there are precisely two irreducible faithful representations of this group of dimension $2$, and they only differ for the character of the two conjugacy classes generated by the elements $\sigma \phi$ and $\sigma^{2} \phi$. To uniquely determine $\psi$ we therefore have to compute the trace of the element $\psi(\sigma \phi)$, and see whether it is $i\sqrt{3}$ or $-i\sqrt{3}$.

We know from Lemma \ref{minmod} that, over $F$, the equation for $E$ is
\begin{align*}
    E: y^2=(x-\alpha)(x-\beta)(x-\gamma)
\end{align*}
and under the change of variables
\begin{align*}
    x=x'(\beta-\alpha)+\alpha;\\
    y=y'(\sqrt{\beta-\alpha})^3,
\end{align*}
we find the minimal model
\begin{align*}
    E_{min}: y^2=x(x-1)(x-\lambda),
\end{align*}
that reduces to 
\begin{align*}
    \tilde E: y^2=x^3-x
\end{align*}
over the residue field. Now let $f(x,y)=(x',y')$ be the change of variables above, $\red$ the reduction map: $E_{min}(\overline K) \rightarrow \tilde{E} (\overline k)$ and $\lift: \tilde E (\overline k) \rightarrow E_{min}(\overline K)$ be any section of $\red$. Then we can compute the action of any Galois automorphism $g$ on the reduced curve $\tilde E (\overline k)$ via the following composition: 
\begin{align*}
    \red \circ f \circ g \circ f^{-1} \circ \lift.
\end{align*}

So in particular for $g=\sigma \Frob$ we have (recall $|k|=3^n$):
\begin{align*}
\left(\tilde{x},\tilde{y}\right) \xrightarrow{\lift}& \left(x,y\right) \xrightarrow{f^{-1}} \left(x(\beta-\alpha) +\alpha,y(\sqrt{\beta-\alpha})^3\right) \\
\xrightarrow{\sigma \Frob}& \left(\sigma(x)^{3^n} \sigma(\beta-\alpha) + \beta, \sigma(y)^{3^n}(\sigma(\sqrt{\beta-\alpha}))^3\right)\\
\xrightarrow{f}& \left(\dfrac{\sigma(x)^{3^n} \sigma(\beta-\alpha) + \beta-\alpha}{\beta-\alpha}, \sigma(y)^{3^n}\dfrac{(\sigma(\sqrt{\beta-\alpha}))^3}{(\sqrt{\beta-\alpha})^3}\right)\\
=&\left( \sigma(x)^{3^n} \dfrac{\sigma(\beta-\alpha)}{\beta-\alpha}+1, \sigma(y)^{3^n}\dfrac{(\sigma(\sqrt{\beta-\alpha}))^3}{(\sqrt{\beta-\alpha})^3}\right) \xrightarrow{\red} \left(\tilde{x}^{3^n}+1,\tilde{y}^{3^n}\right)
\end{align*}

Note that:
\begin{itemize}
    \item the reductions of $\dfrac{\sigma(\beta-\alpha)}{\beta-\alpha} $ and $\dfrac{\sigma(\sqrt{\beta-\alpha})}{\sqrt{\beta-\alpha}} $ are $1$ by Lemma \ref{reduction_wild};
    \item $\Frob$ fixes $F$, therefore it acts trivially on the elements $\beta-\alpha$ and $\sqrt{\beta-\alpha}$;
    \item since $\sigma$ is an inertia element, $\sigma(x) \equiv x$ and $\sigma(y) \equiv y$ in $\overline{k}$.
\end{itemize}    

    Now, to compute the trace of $\rho(\sigma \Frob)$ we use the formula
\begin{align*}
    \tr (\rho(g))=\deg(g) + 1 - \deg(1-g);
\end{align*}
in our case $\deg(g)=\det \rho(\sigma \Frob)=3^n$ and $\deg(1-g)$ is the number of points fixed by $\sigma \Frob$, i.e. the number of solutions (including the point at infinity) of
\begin{align}\label{sys}\left\lbrace
    \begin{array}{l}
         x=x^{3^n}+1  \\
         y=y^{3^n} \\
         y^2=x^3-x.
    \end{array}\right.
\end{align}

Let us first consider the case $n=1$. There are no solutions over $\overline k$ to this system of equations, therefore $\tr(\rho(\sigma \Frob))=3$. But then
\begin{align*}
    \tr(\psi(\sigma \phi))=\dfrac{1}{i\sqrt{3}} 3= -i\sqrt{3}.
\end{align*}

In general, we know that $\tr(\psi(\sigma \Frob))=\varepsilon i\sqrt{3}$ for some $\varepsilon \in \{\pm 1\}$, and $\tr (\rho(\sigma \Frob))=\varepsilon i \sqrt{3} \chi(\sigma \Frob)=\varepsilon i\sqrt{3} i^n \sqrt{3^n}=\varepsilon i^{n+1} \sqrt{3^{n+1}} = \varepsilon (-3)^{(n+1)/2}$. The value of $\varepsilon$ can be determined by solving the system of equations above, but for a general $n$, it cannot be solved directly. However, the number of solutions is independent on the curve we use, so it is sufficient to work with a fixed curve. Consider for example the elliptic curve over $\Q_3$
\begin{align*}
E: y^2 = x^3 +9.
\end{align*}

Since its reduction modulo $3$ is $y^2=x^3$, the valuation of the discriminant is $7$ and the $j$-invariant is $0$, this curve has potentially good reduction, and its N\'eron type is $IV$, so we are in the last case of Theorem \ref{p3}. Hence the image of inertia is isomorphic to $C_3 \rtimes C_4$.

Let us fix a basis for $\overline \Q_\ell^2$ (with $\overline \Q_\ell$ considered as embedded in $\C$), where the action of Frobenius is given by the matrix
\begin{align*}
    \rho(\Frob)= \mat{i\sqrt{3}}{0}{0}{-i\sqrt{3}};
\end{align*}
then the image of $\sigma$ is either $\mat{\zeta_3}{0}{0}{\zeta_3^{-1}}$ or its inverse, where $\zeta_3 \in \overline \Q_\ell$ is the primitive third root of unit $\dfrac{-1+i\sqrt{3}}{2}$. By the computation done for the case $n=1$, we know $\tr(\rho(\sigma \Frob))=3$ and a simple check shows that then $\rho(\sigma)=\mat{\zeta_3^{-1}}{0}{0}{\zeta_3}$. Now let $K$ be an unramified extension of $\Q_3$ of odd degree $n$, so the residue field $k$ is a degree $n$ extension of $\F_3$ and the reduction type and Galois representation of the curve $E$ base-changed to $K$, restricted to inertia is exactly the same as above. Then $\rho(\sigma \Frob)= \mat{\zeta_3^{-1}i^n\sqrt{3^n}}{0}{0}{-\zeta_3i^n\sqrt{3^n}}$, with trace $-(-3)^{(n+1)/2}$. Incidentally, this argument proves the following.
\begin{lemma}
The number of solutions of the system of equations \eqref{sys} above is $3^n + (-3)^{(n+1)/2}$.
\end{lemma}

\begin{proof}
We have $\tr (\rho(\sigma \Frob)) = -(-3)^{(n+1)/2}$. On the other hand, we know
\begin{align*}
    \tr (\rho(\sigma \Frob)) &= |k| + 1 - (1+ |\{\text{solutions to \eqref{sys}}\}|)=\\
    &=3^n - |\{\text{solutions to \eqref{sys}}\}|,
\end{align*}
so the number of solutions to \eqref{sys} is precisely $3^n - \tr(\rho(\sigma \Frob))=3^n +(-3)^{(n+1)/2}$.
\end{proof}

So with the notation above we have $\varepsilon=-1$, and the character of $\psi$ is the following:

$$
\begin{array}{c|ccccccccc}
  \rm class&\rm1&\rm2A&\rm2B&\rm2C&\rm3&\rm4&\rm6A&\rm6B&\rm6C\cr
  \rm size&1&1&2&6&2&6&2&2&2\cr
\hline
  \psi&2&-2&0&0&-1&0&-i\sqrt{3}&i\sqrt{3}&1\cr
\end{array}
$$
as claimed. In particular, in the proof we computed the character of an element of the class $6A$.
\end{proof}

\begin{remark}
Throughout this work, we have considered the Galois representation $\rho$ on the Tate module. It is common to consider instead the representation $\rho_{\acute{e}t}$ on the \'etale cohomology of the elliptic curve. The two representations are dual to each other (see \cite[Theorem 15.1]{corsil}), so for each $g \in G_K$, 
\begin{align*}
    \rho_{\acute{e}t} (g) = (\rho(g)^{-1})^t.
\end{align*}

In this case, it is a convention to consider the Geometric Frobenius Automorphism $\Frob^{-1}$ in place of the Arithmetic Frobenius which we used. Therefore, the eigenvalues of $\rho_{\acute{e}t}(\Frob^{-1})$ are exactly the same as the eigenvalues of the Arithmetic Frobenius under $\rho$. In conclusion, $\rho_{\acute{e}t}$ factors as:
\begin{align*}
    \rho_{\acute{e}t} = \chi_{\acute{e}t} \otimes \psi_{\acute{e}t},
\end{align*}
where
\begin{itemize}
    \item $\chi_{\acute{e}t}= \chi^{-1}$ (therefore $\chi_{\acute{e}t}(\Frob^{-1})=\chi(\Frob)$);
    \item $\psi_{\acute{e}t}$ is the only irreducible representation of $C_3 \rtimes D_4$ with the following character:
    $$
\begin{array}{c|rrrrrrrrr}
  \rm class&\rm1&\rm2A&\rm2B&\rm2C&\rm3&\rm4&\rm6A&\rm6B&\rm6C\cr
  \rm size&1&1&2&6&2&6&2&2&2\cr
\hline
  \psi_{\acute{e}t}&2&-2&0&0&-1&0&i \sqrt{3}&-i \sqrt{3}&1\cr
\end{array}
$$\end{itemize}

So $\psi$ and $\psi_{\acute{e}t}$ only differ for the trace of the conjugacy classes of $\sigma \phi$ and $\sigma^2 \phi$.

A function computing the Galois representation on the \'etale cohomology of an elliptic curve defined over a local field is implemented in MAGMA. The algorithm for the case presented here has recently been improved, using the results in Theorem \ref{mainthm}.
\end{remark}

\subsection*{Acknowledgements}
The author would like to thank her supervisor Tim Dokchitser for the useful conversations and corrections. The work was supported by EPSRC.


\end{document}